\documentclass[bibliography=totoc]{scrartcl}
\usepackage[T1]{fontenc}
\usepackage[utf8]{inputenc}
\usepackage[]{amsmath,amssymb, amsthm}
\usepackage{graphicx}
\usepackage{subfigure}
\usepackage{multirow}
\usepackage[round]{natbib} 
\usepackage{authblk} 
\usepackage{enumerate}
\usepackage{paralist}
\begin{document}
\title{Multivariate Generalized Linear-Statistics of short range dependent data}
\author{Svenja Fischer \thanks{svenja.fischer@rub.de, Faculty for Civil and Environmental Engineering, Ruhr-University Bochum, Germany}
\and Roland Fried \thanks{fried@statistik.tu-dortmund.de, Faculty of Statistics, TU Dortmund University, Germany}
\and Martin Wendler\thanks{martin.wendler@rub.de, Faculty of Mathematics, Ruhr-University Bochum, Germany }}

\date{}
\maketitle
\bibliographystyle{plainnat}

\setlength{\parindent}{0pt}
\allowdisplaybreaks

\newtheorem{definition}{Definition}[section]
\newtheorem{theorem}{Theorem}[section]
\newtheorem{corollary}{Corollary}[section]
\newtheorem{lemma}{Lemma}[section]
\newtheorem{proposition}{Proposition}[section]

\theoremstyle{definition}
\newtheorem{example}{Example}[section]
\newtheorem{remark}{Remark}[section]

\newcommand{\var}{\operatorname{Var}}
\newcommand{\cov}{\operatorname{Cov}}
\newcommand{\med}{\operatorname{med}}

\begin{abstract}

\noindent
Generalized linear ($GL$-) statistics are defined as functionals of an $U$-quantile process and unify different classes of statistics such as $U$-statistics and $L$-statistics. We derive a central limit theorem for $GL$-statistics of strongly mixing sequences and arbitrary dimension of the underlying kernel. For this purpose we establish a limit theorem for $U$-statistics and an invariance principle for $U$-processes together with a convergence rate for the remaining term of the Bahadur representation. 

\noindent
An application is given by the generalized median estimator for the tail-parameter of the Pareto distribution, which is commonly used to model exceedances of high thresholds.
 We use subsampling to calculate confidence intervals and investigate its behaviour under independence and strong mixing in simulations.


\noindent
MSC 62G30 \and MSC 60G10 \and MSC 60F1
\end{abstract}

\section{Introduction}
\label{Intro}

Generalized linear statistics ($GL$-statistics) form a broad class of statistics, which unifies not only the widely used $U$-statistics but also other classes like $L$-statistics and even statistics which cannot be assigned to a certain class. $GL$-statistics were first developed by \cite{Serf1984}, who shows a central limit theorem under independence. In this paper we develop results for $GL$-statistics of random variables which are short range dependent. An important tool to gain a Central Limit Theorem for $GL$-statistics are $U$-statistics with multivariate kernels. Up to now we can find a lot of results for bivariate $U$-statistics of short range dependent data (cf. \cite{Boro}, \cite{Deh2010} and \cite{Wen2011}) but in the multivariate case there occur some additional difficulties caused by the dependencies in the kernel structure. 

\vspace*{0.1cm}

\vspace*{0.3cm}
Now let us introduce some basic assumptions and definitions which we will use throughout the paper.

Let $X_1,\ldots, X_n$ be a sequence of random variables with distribution function  $F$. We will assume the random variables to be short range dependent, a detailed definition is given later on. Moreover, let $F_n$ be the empirical distribution function of $X_1,\ldots,X_n$ with
\begin{align*}
F_n(x)=\frac{1}{n}\sum_{i=1}^{n}{1_{\left[X_i\leq x\right]}}, ~-\infty<x<\infty,
\end{align*}
and $h(x_1,\ldots, x_m)$, for given $m\geq 2$, a kernel, that is a measurable, symmetric function. We define the empirical distribution function $H_n$ of $h\left(X_{i_1},\ldots, X_{i_m}\right)$ as
\begin{align*}
H_n(x)=\frac{1}{\binom{n}{m}}\sum_{1\leq i_1<\ldots < i_m\leq n}{1_{\left[h\left(X_{i_1},\ldots , X_{i_m}\right)\leq x\right]}}, ~ -\infty<x<\infty
\end{align*}
and $H_n^{-1}(p)=\inf\lbrace x \vert H_n(x)\geq p\rbrace$ as the related generalized inverse.
Furthermore, let $H_F$ with $H_F(y)=\mathbb{P}_F(h(Y_1,\ldots,Y_m)\leq y)$ be the distribution function of the kernel $h$ for independent copies $Y_1,\ldots,Y_m$ of $X_1$ and $0<h_F<\infty$ the related density (this implies that $H_F$ is continuous).

We define $h_{F;X_{i_2},\ldots,X_{i_k}}$ as the density of $h(Y_{i_1},X_{i_2},\ldots,X_{i_k},Y_{i_{k+1}},\ldots,Y_{i_m})$ for $2\leq k \leq m$ and $i_1<i_2<\ldots<i_m$.
\vspace*{0.15cm}

\begin{definition}
$\newline$
A generalized $L$-statistic with kernel $h$ is given by
\begin{align*}
T(H_n)&=\int_0^1{H_n^{-1}(t)J(t)dt}+\sum_{i=1}^d{a_iH_n^{-1}(p_i)}\\
&=\sum_{i=1}^{n_{(m)}}{\left[\int_{\frac{(i-1)}{n_{(m)}}}^{\frac{i}{n_{(m)}}}{J(t)dt}\right]H_n^{-1}\left(\frac{i}{n_{(m)}}\right)}+\sum_{i=1}^d{a_iH_n^{-1}(p_i)}.
\end{align*}
\end{definition}
The $GL$-statistic $T(H_n)$ is a natural estimator of $T(H_F)$, which is defined analogously.

\vspace*{0.5cm}
\begin{example}\label{Examu}
$\newline$
Let $h:\mathbb{R}^m\rightarrow \mathbb{R}$ be a measurable function.
A $U$-statistic with kernel $h$ is defined as $$U_n=\frac{1}{\binom{n}{m}}\sum_{1\leq i_1<\ldots < i_m\leq n}{h(X_{i_1},\ldots,X_{i_m})}.$$
If the random variables are independent and identically distributed, $U_n$ is an unbiased estimator of $\theta=\mathbb{E}(h(X_1,\ldots,X_m))$. 
A $U$-statistic can be written as a $GL$-statistic by setting $d=0$ and $J=1$.
\end{example}

\begin{example}
$\newline$
A widely known $L$-statistic is the \textit{$\alpha$-trimmed mean} 
\begin{align*}
\bar{X}_{(\alpha)}=\frac{1}{n-2\left[n\alpha \right]}\sum_{i=\left[n\alpha\right]+1}^{n-\left[n\alpha\right]}{X_{(i)}},
\end{align*}
where $X_{(i)}$ is the $i$th value of the order statistic $X_{(1)}\leq X_{(2)}\leq \ldots \leq X_{(n)}$.
To rewrite it as a $GL$-statistic we choose $J(t)=\tfrac{1}{1-2\alpha}$ for $\alpha<t<1-\alpha$ and $J(t)=0$ everywhere else. As kernel we set $h(x)=x$ and let the sum vanish by the choice $d=0$.
\end{example}

\begin{example}
$\newline$
The generalized \textit{Hodges-Lehmann estimator} 
\begin{align*}
\operatorname{median} \left(\frac{1}{2}\left(X_{i_1}+\ldots+X_{i_m}\right),~ 1\leq i_1,\ldots, i_m \leq n\right)
\end{align*}
is neither an $U$-statistic nor a $L$-statistic, but it is possible to formulate it as a $GL$-statistic choosing the kernel $h(x_{i_1},\ldots,x_{i_m})=\tfrac{1}{2}(x_{i_1}+\ldots+x_{i_m})$ and setting $J=0$, $d=1$, and $a_1=1$. We get the median of the kernel by using the representation via the quantile function $H_n^{-1}(\tfrac{1}{2})$. Consequently $p_1=\tfrac{1}{2}$. The generalized Hodges-Lehmann estimator is the $GL$-statistic  
\begin{align*}
T(H_n)=H_n^{-1}\left(\frac{1}{2}\right).
\end{align*}
\end{example}

\vspace*{0.2cm}

In the following we will consider a special form of short range dependence: strong mixing.

\begin{definition}

Let $(X_n)_{n\in\mathbb{N}}$ be a stationary process. The strong mixing coefficients of $(X_n)$ are 
\begin{align*}
\alpha(k)=\sup_{n\in\mathbb{N}}{\sup{\left\{\left\vert \mathbb{P}(A\cap B)-\mathbb{P}(A)\mathbb{P}(B)\right\vert:A\in\mathcal{F}_1^n,B\in\mathcal{F}_{n+k}^\infty\right\}}},
\end{align*}
where $\mathcal{F}_a^b$ is the $\sigma$-field  generated by $X_a,\ldots,X_b$.

$(X_n)_{n\in\mathbb{N}}$ is called strongly mixing (or $\alpha$-mixing), if $\alpha(k)\rightarrow0$ as $k\rightarrow \infty$.
\end{definition}
Strong mixing is the weakest among the different forms of mixing since the $\alpha$-mixing coefficients are always smaller than for example the $\beta$-mixing coefficients (cf. \cite{Brad2007}).

\vspace*{0.2cm}

After stating the main results, among others the Central Limit Theorem for GL-statistics, we also provide some results concerning $U$-statistics and $U$-processes. In a second step we give an application, the generalized median estimator ($GM$-estimator) for the tail parameter of the Pareto distribution
(cf. \cite{Bra2000} and \cite{Bra22000} under independence). 
The Pareto distribution is commonly used for modelling heavy tails and exceedances of a threshold (peak over threshold, POT). Especially in hydrology it has wide application when only extreme floods above a certain threshold should be considered in the analysis. There also occurs the need of a robust estimator, needing a downweighting of the influence of extreme floods in short time series. Simulations verify that the generalized median estimator is almost as efficient as the  maximum likelihood estimator under independence and for autocorrelated data, but more robust. 
Short range dependence is up to now seldom modelled in the estimation of parameters under POT, but when considering for example monthly discharges it is very probable to find such dependencies. Our investigation of the generalized median estimator aims at closing this gap and can be extended to other situations, where a robust estimator for dependent data is needed.

\vspace*{0.1cm}

Results needed for the proofs of the main results are given in Section \ref{sec:6}, the proofs of the results given in Section \ref{sec:2} can be found in Section \ref{sec:7}.
\vspace*{0.2cm}

\section{Main Results}
\label{sec:2}

An important and well known result concerning quantiles is the representation proposed by Bahadur, which uses the representation of the quantile by the empirical distribution function. A key role plays the remaining term, for which \cite{Gho1971} showed the convergence for ordinary quantiles and under independence. In our case we need the convergence of generalized quantiles and strong mixing. The result is stated in the following theorem.

\begin{theorem}\label{bahad}
$\newline$
Let $(X_n)_{n\in\mathbb{N}}$ be a sequence of strong mixing random variables with distribution function $F$, $\mathbb{E}\lvert X_1 \rvert ^{\rho}<\infty$ for a $\rho\geq1$ and mixing coefficients $\alpha(l)=O(l^{-\delta})$ for a $\delta>\frac{2\rho+1}{\rho}$.
 Moreover let $h(x_1,\ldots,x_m)$ be a Lipschitz-continuous kernel with distribution function $H_F$ and related density $0<h_F<\infty$ and for all $2\leq k \leq m$ let $h_{F;X_2,\ldots,X_k}$ be bounded. Then we have for the Bahadur representation with $\hat{\xi}_p=H_n^{-1}(p)$
\begin{align*}
\hat{\xi}_p=\xi_p+\frac{H_F(\xi_p)-H_n(\xi_p)}{h_F(\xi_p)}+o_p(\frac{1}{\sqrt{n}}).
\end{align*}
\end{theorem}

Now we will state the main theorem of our paper, the asymptotic normality of $GL$-statistics under strong mixing. Under independence this result was proved by \cite{Serf1984}.

\begin{theorem}\label{main}
$\newline$
Let $h(x_1,\ldots,x_m)$ be a Lipschitz-continuous kernel with distribution function
$H_F$ and related density $0<h_F <\infty$ and for all $2\leq k \leq m$ and all $i_1<i_2<\ldots<i_m$ let $h_{F;X_{i_2},\ldots,X_{i_k}}$ be bounded. Moreover let $J$ be a function with $J(t)=0$ for $t\notin\left[\alpha,\beta\right]$ , $0<\alpha<\beta<1$, and in $\left[\alpha,\beta\right]$ let $J$ be bounded and a.e. continuous concerning the Lebesgue-measure and a.e. continuous concerning $H_F^{-1}$.
Additionally, let $X_1,\ldots,X_n$ be a sequence of strong mixing random variables with $\mathbb{E}\lvert X_1 \rvert^{\rho}<\infty$  for a $\rho\geq 1$ and mixing coefficients $\alpha(n)$ with  $\alpha(n)=O(n^{-\delta})$ for a $\delta \geq 8$. 
Then the following statement holds for $GL$-Statistics $T(H_n)$
\begin{align*}
\sqrt{n}\left(T(H_n)-T(H_F)\right)\stackrel{\mathcal{D}}{\longrightarrow}N(0,\sigma^2),
\end{align*}
where 
\begin{align*}
\sigma^2=&m^2\big(\var\left(\mathbb{E}\left(A(Y_1,\ldots,Y_{m})\vert Y_1=X_1\right)\right)\\
&+2\sum_{j=1}^\infty \cov\left( \mathbb{E}\left(A(Y_1,\ldots,Y_{m})\vert Y_1=X_1\right),\mathbb{E}\left(A(Y_1,\ldots,Y_{m})\vert Y_{1}=X_{j+1}\right)\right)\big)
\end{align*}
with independent copies $Y_1,\ldots,Y_m$ of $X_1$
and
\begin{align*}
A(x_1,\ldots,x_m)=&-\int_{-\infty}^{\infty}{\left(1_{\left[h(x_1,\ldots,x_m)\leq y\right]}-H_F(y)\right)J(H_F(y))dy}\\
&+\sum_{i=1}^{d}{a_i\frac{p_i-1_{\left[h(x_1,\ldots,x_m)\leq H_F^{-1}(p_i)\right]}}{h_F(H_F^{-1}(p_i))}}.
\end{align*}
\end{theorem}

For the proof of this theorem, which is given in Section \ref{sec:7}, a key tool will be the representation of the kernel $A$ as a $U$-statistic, see Example \ref{Examu}. Additionally also the functional $H_n$ belongs to the class of $U$-statistics and therefore we make use of several results of the theory of $U$-statistics. In the following section we will extend some known results for bivariate $U$-statistics under strong mixing to the multivariate case. We will see that this extension causes some problems concerning the dependencies in the kernels and the solution of these problems is not straightforward.

\begin{remark}
In the case of bivariate kernels, similar results as Theorems \ref{asynom} and \ref{invpri} can be found in  \cite{Boro}, \cite{Deh2002} and \cite{Wen2011} for NED-sequences of absolutely regular processes. We conjecture that an extension to the multivariate case is possible also under this other type of weak dependence, but detailed proofs are beyond the scope of this paper.
\end{remark}

\subsection{$U$-statistics and $U$-processes}
\label{sec:3}

While examining $U$-statistics often a technique called Hoeffding decomposition (\cite{Hoeff1948}) is used. It decomposes the $U$-statistic into a sum of different terms, which we can examine separately.

\begin{definition}\label{hoeff}(Hoeffding decomposition)
$\newline$
Let $U_n$ be a $U$-statistic with kernel $h=h(x_1,\ldots,x_m)$. Then one can write $U_n$ as
\begin{align*}
U_n=\theta+\sum_{j=1}^{m}{\binom{m}{j}\frac{1}{\binom{n}{j}}S_{jn}},
\end{align*}
where 
\begin{align*}
\theta&=\mathbb{E}(h(Y_1,\ldots,Y_m))\\
\tilde{h}_j(x_1,\ldots,x_j)&=\mathbb{E}(h(x_1,\ldots,x_j,Y_{j+1},\ldots,Y_m))-\theta\\
S_{jn}&=\sum_{1\leq i_1<\ldots<i_j\leq n}g_j(X_{i_1},\ldots,X_{i_j})\\
g_1(x_1)&=\tilde{h}_1(x_1)\\
g_2(x_1,x_2)&=\tilde{h}_2(x_1,x_2)-g_1(x_1)-g_1(x_2)\\
g_3(x_1,x_2,x_3)&=\tilde{h}_3(x_1,x_2,x_3)-\sum_{i=1}^{3}g_1(x_i)-\sum_{1\leq i<j\leq 3}g_2(x_i,x_j)\\
&\ldots\\
g_m(x_1,\ldots,x_m)&=\tilde{h}_m(x_1,\ldots,x_m)-\sum_{i=1}^{m}g_1(x_i)-\sum_{1\leq i_1<i_2\leq m}g_2(x_{i_1},x_{i_2})\\
&-\ldots - \sum_{1\leq i_1<\ldots<i_{m-1}\leq m}{g_{m-1}(x_{i_1},\ldots,x_{i_{m-1}})}.
\end{align*}
for independent copies $Y_1,\ldots,Y_m$ of $X_1$.
\end{definition}

The term $\frac{m}{n}\sum_{i=1}^ng_1(X_i)$ is called the linear part, the remaining parts are called degenerated.

\vspace*{0.4cm}

For most of the results in this section we need a regularity condition for the kernel $h$, which was first developed by \cite{Denk1986} and is extended for our purpose.

\begin{definition}
A kernel $h$ satisfies the variation condition, if there exists a constant $L$ and an $\epsilon_0>0$, such that for all $\epsilon\in (0,\epsilon_0)$
\begin{align*}
\mathbb{E}\left(\sup\limits_{\lVert(x_1,\ldots,x_m)-(X_1',\ldots,X_m')\rVert\leq \epsilon}\left| h(x_1,\ldots,x_m)-h(X_1',\ldots,X_m')\right|\right)\leq L\epsilon,
\end{align*}
where the $X_i'$ are independent with the same distribution as $X_i$ and $\lVert \cdot \rVert$ is the Euklidean norm. 

A kernel $h$ satisfies the extended variation condition, if there additionally exist constants $L'>0$ and $\delta_0>0$, such that for all $\delta\in(0,\delta_0)$ and all $2\leq k \leq m$
\begin{align*}
&\mathbb{E}\Biggl(\sup\limits_{\lvert x_{i_1}-Y_{i_1}\rvert\leq \delta}\left| h(x_{i_1},X_{i_2},\ldots,X_{i_k},Y_{i_{k+1}},\ldots,Y_{i_m})\right.\\
&\phantom{\mathbb{E}\left(\sup\limits_{\lvert x_{i_1}-Y_{i_1}\rvert\leq\delta}\right. h(x_{i_1},X_{i_2},\ldots,)}\left.-h(Y_{i_1},X_{i_2},\ldots,X_{i_k},Y_{i_{k+1}},\ldots,Y_{i_m})\right|\Biggr)\\
&\leq L'\delta
\end{align*} 
for independent copies $(Y_n)_{n\in\mathbb{N}}$ of $(X_n)_{n\in\mathbb{N}}$ and all $i_1<i_2<\ldots <i_m$. 
If the kernel has dimension $m=1$, we note that it satisfies the extended variation condition, if it satisfies the variation condition.
\end{definition}

\begin{remark}
Every Lipschitz-continuous kernel satisfies the variation condition.
\end{remark}

Now we state another main result of this paper, the aymptotic normality of $U$-statistics under strong mixing. For bivariate $U$-statistics this result is already known (see \cite{Wen2011}), but not for arbitrary dimension $m$ of the kernel $h$.

\begin{theorem}\label{asynom}
$\newline$
Let $h:\mathbb{R}^m\rightarrow \mathbb{R}$ be a bounded kernel satisfying the extended variation condition. 
Moreover let $(X_n)_{n\in\mathbb{N}}$ be a sequence of strong mixing random variables with $\mathbb{E}\lvert X_1\rvert^{\rho}<\infty$ for a $\rho>0$ and mixing coefficients $\alpha(l)=O(l^{-\delta})$ for a $\delta>\tfrac{2\rho+1}{\rho}$. 
Then we have
\begin{align*}
\sqrt{n}(U_n-\theta)\stackrel{D}{\longrightarrow}N(0,m^2\sigma^2)
\end{align*}
with $\sigma^2=\var(g_1(X_1))+2\sum_{j=1}^{\infty}{\cov(g_1(X_1),g_1(X_{1+j}))}$.

If $\sigma=0$ then the statement means convergence to $0$ in probability.
\end{theorem}

The key tool for the proof of this theorem is the Hoeffding decomposition, for which the first term converges against the given distribution while all remaining terms converge towards zero.

\vspace*{0.5cm}

As an extension to $U$-statistics we also analyse $U$-processes and their convergence. In other words our $U$-statistic has no longer a fixed kernel $h$ but we have a process $(U_n(t))_{t\in\mathbb{R}}$. Up to now we have had $(H_n(t))_{t\in\mathbb{R}}$ as an example of such a process.

\vspace*{0.1cm}

\begin{definition}
Let $h:\mathbb{R}^{m+1}\rightarrow\mathbb{R}$ be a measurable and bounded function, symmetric in the first $m$ arguments and non-decreasing in the last. Suppose that for all $x_1,\ldots,x_m \in \mathbb{R}$ we have $\lim\limits_{t\rightarrow \infty}h(x_1,\ldots,x_m,t)=1,~\lim\limits_{t\rightarrow -\infty}h(x_1,\ldots,x_m,t)=0.$
We call the process $(U_n(t))_{t\in \mathbb{R}}$ empirical $U$-distribution function. As $U$-distribution function we define $U(t):=\mathbb{E}\left(h(Y_1,\ldots,Y_m,t)\right)$ for independent copies $Y_1,\ldots,Y_m$ of $X_1$. Then the empirical process is defined as 

$$(\sqrt{n}(U_n(t)-U(t)))_{t \in \mathbb{R}}.$$
\end{definition}

\vspace*{0.2cm}

Analogous to simple $U$-statistics here the Hoeffding decomposition is an important technique in our proofs. 
For fixed $t$ we have
\begin{align*}
U_n(t)=\frac{1}{\binom{n}{m}}\sum_{1\leq i_1<\ldots<i_m\leq n}h(X_{i_1},\ldots,X_{i_m},t)
\end{align*}
 and therefore we can decompose $U_n(t)$ analogously to Definition \ref{hoeff}.

Likewise we will need a new form of the extended variation condition. 
\begin{definition}
We say $h$ satisfies the extended uniform variation condition, if the extended variation condition holds for $h(x_1,\ldots,x_m,t)$ with a constant not depending on $t$.
\end{definition}

\vspace*{0.3cm}

A typical result for processes is the Invariance Principle, a result we also need for our $U$-processes. For near epoch dependent sequences on absolutely regular processses it was already proved by \cite{Deh2002}. A result for strong mixing can be found in \cite{Wen2011}. Under independence one can find a strong invariance principle in \cite{De1987}. Nevertheless these results only consider the bivariate case, whereas we also admit multivariate kernels. For our purposes we only need the convergence of the first term of the Hoeffding decomposition, so the proof will be somewhat different. 

\vspace*{0.1cm}
From now on consider the case where $H_n$ is our empirical $U$-process, that is $U_n(t)$ has the kernel $g(x_1,\ldots,x_m,t)=1_{[h(x_1,\ldots,x_m)\leq t]}$. Therefore $U(t)=\mathbb{E}\left(1_{[h(Y_1,\ldots,Y_m)\leq t]}\right)=\mathbb{P}(h(Y_1,\ldots,Y_m)\leq t)=H_F(t)$ and since $H_F$ has density $h_F<\infty$ we have that $H_F$ is Lipschitz-continuous.

\begin{theorem}\label{invpri}
$\newline$
Let $h$ be a kernel with distribution function $H_F$ and related density $h_F<\infty$.
Moreover, let $g_1$ be the first term of the Hoeffding decomposition of $H_n$.
Let $(X_n)_{n\in\mathbb{N}}$ be a sequence of strong mixing random variables with mixing coefficients $\alpha(l)=O(l^{-6-\gamma})$ for a $0<\gamma <1$.
Then
\begin{align*}
\left(\frac{m}{\sqrt{n}}\sum_{i=1}^n g_1(X_i,t)\right)_{t\in\mathbb{R}}\stackrel{D}{\longrightarrow}\left(W(t)\right)_{t\in \mathbb{R}},
\end{align*}
where $W$ is a continuous Gaussian process.
\end{theorem}

This theorem can be proved in the same way as Theorem 4.1 of \cite{Deh2002} and is therefore omitted.

\vspace*{0.15cm}

By using results concerning the convergence of all remaining terms of the Hoeffding decomposition, which is given in Lemma \ref{remterm}, we can state the following corollary.

\begin{corollary}\label{glican}
$\newline$
Let $(X_n)_{n\in\mathbb{N}}$ be a sequence of strong mixing random variables with mixing coefficients $\alpha(l)=O(l^{-\delta})$ for $\delta\geq 8$ and $\mathbb{E}\lvert X_1\rvert^{\rho}<\infty$ for a $\rho>\frac{1}{4}$.
Moreover let $h$ be a Lipschitz-continuous kernel with distribution function $H_F$ and related density $h_F<\infty$ and for all $2\leq k \leq m$ let $h_{F;X_2,\ldots,X_k}$ be bounded.
Then
\begin{align*}
\sup\limits_{t\in\mathbb{R}}\left\vert\sqrt{n}\left(H_n(t)-H_F(t)\right)\right\vert=O_p(1).
\end{align*}
\end{corollary}

The proofs of all results in this section are given in Section \ref{sec:7}.

\section{Application: The Generalized Median Estimator}
\label{sec:4}

The generalized median ($GM$-) estimator was developed by Brazauskas and Serfling under independence as a robust estimator of the parameters of different distributions, for example the Pareto distribution or Log-Normal distribution (\cite{Bra2000}, \cite{Bra22000} and \cite{Serf2002}).

We will concentrate on the Pareto distribution, which is a very heavy tailed distribution often used in hydrology and other fields for modelling the tail of a distribution. Its distribution function is given by
\begin{align*}
F(x)=\begin{cases} 1-\left(\frac{\sigma}{x}\right)^\alpha, ~& x\geq \sigma \\
0, & x<\sigma
\end{cases},
\end{align*}
where $\alpha>0$ and $\sigma>0$. We assume $\sigma$ to be unknown and estimate it through the minimum of the sample.

We want to expand the $GM$-estimator to sequences of strong mixing random variables with Pareto distributed margins and estimate the tail index $\alpha$. Therefore we have to choose a kernel which is median unbiased. Like \cite{Bra2000} we choose the modified maximum likelihood estimator as kernel, which was shown to be median unbiased under independence, and use this result to show its asymptotical median unbiasedness under strong mixing, that is \[h(x_1,\ldots,x_m)=\frac{M_{2m-2}}{2m}\frac{1}{\left(\frac{1}{m}\sum_{i=1}^m\log x_i-\log \left(\min\left((x_1,\ldots,x_m \right)\right) \right)},\] where $M_{2m-2}$ is the median of the $\chi_{2m-2}^2$-distribution.

\begin{lemma}
For a sequence of strong mixing, Pareto distributed random variables $(X_n)_{n\in\mathbb{N}}$ with $\mathbb{E}\left\vert X_1\right\vert^{\rho}<\infty$ for a $\rho\geq 1$ and mixing coefficients $\alpha(l)=O(l^{-\delta})$ for a $\delta\geq 8$ the kernel

 $h(x_1,\ldots,x_m)=\frac{M_{2m-2}}{2m}\frac{1}{\left(\frac{1}{m}\sum_{i=1}^m\log x_i-\log \left(\min\left(x_1,\ldots,x_m \right)\right) \right)}$ is asymptotically median unbiased.
\end{lemma}

\begin{proof}
$\newline$
We have $\mathbb{E}\left(H_n-H_F\right)^2\longrightarrow 0$ using the same arguments as in Lemma \ref{bahad}. With arguments of Glivenko-Cantelli type this implies 
\begin{align*}
\sup\limits_{t}\left\vert \mathbb{E}(H_n(t))-H_F(t)\right\vert\longrightarrow 0.
\end{align*}
Following Example 1 of \cite{Pol1984} the proof is completed. \qed
\end{proof}
\vspace*{0.1cm}

The $GM$-estimator of the parameter $\alpha$ is then given by
\begin{align*}
\hat{\alpha}_{GM}=\med(h(X_{i_1},\ldots,X_{i_m})),
\end{align*}
which can be expressed as an $GL$-statistic by choosing $J=0,~d=1,~a_1=1$, and $p_1=\tfrac{1}{2}$. 
Applying Theorem \ref{main} we have
\begin{align*}
\sqrt{N}\left(\hat{\alpha}_{GM}-\alpha\right)\stackrel{D}{\longrightarrow}N(0,\sigma^2_{GM}),
\end{align*}
\begin{align*}
\sigma^2_{GM}=&\frac{m^2}{h_F^2(\alpha)}\Bigg(Var(\mathbb{P}(h(Y_1,\ldots,Y_m)\leq \alpha\vert Y_1=X_1))\\
&+2\sum_{j=1}^\infty Cov\bigl(\mathbb{P}(h(Y_1,\ldots,Y_m)\leq \alpha\vert Y_1=X_1),\\
&\phantom{+2\sum_{j=1}^\infty Cov(\mathbb{P})}\mathbb{P}(h(Y_1,\ldots,Y_m)\leq \alpha\vert Y_1=X_{j+1})\bigr)\Bigg).
\end{align*}

The results concerning robustness given by \cite{Bra22000} remain valid since the kernel is unchanged. Additionally one can show that the influence function of the $GM$-estimator is bounded (cf. \cite{Serf1984}).

\vspace{0.3cm}

In the following simulations we compute confidence intervals for the tail index $\alpha$ using subsampling (cf. \cite{Poli1994}). We show the coverage probability and the length of the confidence interval for different block lengths in subsampling and three different kernel dimensions of the generalized median estimator, that is $m=2,3,4$. The underlying n=100 random variables we compute as independent, identically Pareto-distributed with $\alpha=1$ and $\sigma=2$ and also from an AR(1)-process with autocorrelation coefficient $\rho=0.2$ and Pareto-distributed margins. The simulation is repeated 500 times. 
\vspace{0.1cm}
The procedure of subsampling is as follows:

Because $\sqrt{n}\left(\hat{\alpha}_{GM}-\alpha\right)$ converges against an unknown distribution, we estimate the quantiles of the distribution the following way: we first choose a blocklength $b=b_n$ with $b_n\rightarrow \infty$ and $\tfrac{b_n}{n}\rightarrow 0$ for $n\rightarrow \infty$. Then we calculate the $GM$-estimator of $\alpha$ for each of the $n-b+1$ subsamples consisting of $b$ consecutive data values, getting a vector of estimates $\left(\hat{\alpha}_{GM}^1,\ldots,\hat{\alpha}_{GM}^{n-b+1}\right)$. Using 
\begin{align*}
L_n(t)=\frac{1}{n-b+1}\sum_{i=1}^{n-b+1}1_{[\sqrt{b}\left(\hat{\alpha}_{GM}^i-\hat{\alpha}\right)]\leq t}
\end{align*}
the quantiles $q^\ast_\gamma=L_n^{-1}(\gamma)$ are calculated, where $\hat{\alpha}$ is the $GM$-estimate for $\alpha$ derived from the whole sample.

The confidence interval CI for a confidence level $1-\gamma$ is then
\begin{align*}
CI=\left[\hat{\alpha}-\frac{q^\ast_{1-\tfrac{\gamma}{2}}}{\sqrt{n}};\hat{\alpha}-\frac{q^\ast_{\tfrac{\gamma}{2}}}{\sqrt{n}}\right],
\end{align*}
resulting from $\mathbb{P}\left(q^\ast_{\tfrac{\gamma}{2}}\leq \sqrt{n}\left(\hat{\alpha}-\alpha\right)\leq q^\ast_{1-\tfrac{\gamma}{2}}\right)\longrightarrow 1-\gamma$.

These results are compared with the case $m=n$ corresponding to the maximum-likelihood (ML) estimator.

All simulations were done in R 3.0.1 using the packages VGAM and fExtremes and the algorithm of \cite{Wild2013} for the generalized median estimator. We need to mention that the results can fluctuate up to 0.02 because of the moderate number of observation runs (500).

\begin{table}
\caption{Confidence interval length and coverage probability of the 90\% and 95\% confidence intervals from 100 independent, identically Pareto(2,1)-distributed random variables using subsampling under different block lengths for 500 repetitions and kernel dimension 2, 3, 4 and n}
	\centering

		\begin{tabular}[c]{| c | c | c | c | c |}
		\hline
			\multirow{2}{*}{block length} & \multicolumn{2}{| c |}{90\% confidence interval} & \multicolumn{2}{| c |}{95\% confidence interval}\\
			\cline{2-5}
			  & coverage probability & length &coverage probability & length \\
			\hline
		\multicolumn{5}{|c|}{m=2}\\
			  \hline
				15 & 0.776 & 0.769 & 0.848 & 0.894\\
				\hline
				20 & 0.738 & 0.701& 0.818 & 0.795 \\
				\hline
				\hline
		\multicolumn{5}{|c|}{m=3}\\
				\hline
				15 & 0.778& 0.736& 0.812 & 0.845 \\
				\hline
				20 &0.770 &0.674 &0.792 &0.738 \\
				\hline
				\hline
		\multicolumn{5}{|c|}{m=4}\\
				\hline
				15 &0.781 &0.720 &0.843  &0.814 \\
				\hline
				20 &0.772 &0.683 &0.805&0.697 \\
				\hline
				\hline
		\multicolumn{5}{|c|}{m=n}\\
				\hline
				15 &0.834 &0.666 &0.846  &0.734 \\
				\hline
				20 &0.792 &0.585 &0.818 &0.658 \\
				\hline
		\end{tabular}
	\label{tab:Cbounds1}

\end{table}

\begin{table}
\caption{Confidence interval length and coverage probability of the 90\% and 95\% confidence intervals from 100 random variables from an AR(1)-process with $\rho=0.2$ and  Pareto(2,1)-distributed margins using subsampling under different block lengths for 500 repetitions and kernel dimension 2, 3, 4 and n}
	\centering

		\begin{tabular}[c]{| c | c | c | c | c |}
		\hline
			\multirow{2}{*}{block length} & \multicolumn{2}{| c |}{90\% confidence interval} & \multicolumn{2}{| c |}{95\% confidence interval}\\
			\cline{2-5}
			  & coverage probability & length &coverage probability & length \\
		\hline
		\multicolumn{5}{|c|}{m=2}\\
				\hline
				15 &0.756 &0.874 &0.778 &1.005 \\
				\hline
				20 & 0.756 &0.789 &0.770 &0.878 \\
				\hline
				\hline
		\multicolumn{5}{|c|}{m=3}\\
				\hline
				15 &0.794 &0.850& 0.764&0.950 \\
				\hline
				20 & 0.724&0.779 & 0.780&0.864\\
				\hline
				\hline
		\multicolumn{5}{|c|}{m=4}\\
				\hline
				15 &0.803 &0.838 &0.811 &0.943 \\
				\hline
				20 &0.769 &0.744 &0.776 &0.822 \\
				\hline
				\hline
		\multicolumn{5}{|c|}{m=n}\\
				\hline
				15 &0.790 &0.840 &0.814  &0.994 \\
				\hline
				20 &0.770 &0.749 &0.796 &0.853 \\
				\hline
		\end{tabular}
	\label{tab:Cbounds2}
	
\end{table}

First we investigate the efficiency of the $GM$-estimator in comparison with the classical maximum-likelihood estimator corresponding to the case $m=n$. For this we have a look at the coverage probability and the length of the confidence interval under data from an ideal model.
As expected we see in Tables \ref{tab:Cbounds1} and \ref{tab:Cbounds2} that under independence the coverage probability and the length of the confidence interval of the $GM$-estimator get better for increasing $m$, being best when $m=n$, the case of the $ML$-estimator. Nevertheless even for small values of $m$ the efficiency of the $GM$-estimator is close to that of the $ML$-estimator. 

Under slight dependence ($\rho=0.2$) the $GM$-estimator with $m=4$ performs almost as well as the $ML$-estimator with $m=n$ and the length of the confidence interval is sometimes even smaller. Note that in the case of dependence, the $GM$-estimator for $m=n$ is not the $ML$-estimator, since it was constructed to maximize the likelihood under independence. Nevertheless, this estimator for $m=n$ is widely applied also under dependence and we use it for comparison.
In general the coverage probability and also the length of the confidence interval of the $GM$-estimator are not influenced very much by the size of $m$; for the smallest choice of $m$ the coverage probability and the length of the confidence interval of the $GM$-estimator are rather close to that of the case $m=n$.

For independence or moderate dependence ($\rho=0.2$), the coverage probability decreases  when the block length $b$ increases. For stronger dependence ($\rho=0.8$), the longer block length ($b=20$) gives better results.

We also tested the case where $\rho=0.8$, but the results for a sample size $n=100$ were very poor for all cases of $m$ with a coverage probability always about 0.3 and a length of the confidence interval between 3 and 10, and therefore they are omitted here.

Additionally we compared the robustness of the $ML$-estimator ($m=n$) with the $GM$-estimator for $m=2$, the most robust case. We contaminate a sample by adding a value $y_i$ of the interval $(0,100]$, and calculate the average coverage probability, that is 
\begin{align*}
{CP}(i)=\frac{1}{n}\sum_{j=1}^n\left(1_{[CI1_j,CI2_j]}(\alpha)-1_{[CI1_j(i),CI2_j(i)]}(\alpha)\right),
\end{align*}
where $CI1_j$ and $CI2_j$ are the bounds of the confidence interval calculated for the $jth$ sample $(X_j^{(1)},\ldots,X_j^{(n)})$ and $CI1_j(i)$ and $CI2_j(i)$ are the bounds of the confidence interval calculated for the $jth$ sample contaminated by $y_i$, $(X_j^{(1)},\ldots,X_j^{(n)},y_i)$, for a confidence level of $0.95$ respectively and $j=1,\ldots,100$. The confidence intervals were again computed by subsampling with a block length of 15. This method is analogous to classical sensitivity curves, but focuses on the coverage probability. The results can be found in Figure \ref{fig:SenCurve}.

\begin{figure}
	\centering
		\includegraphics[width=1.00\textwidth]{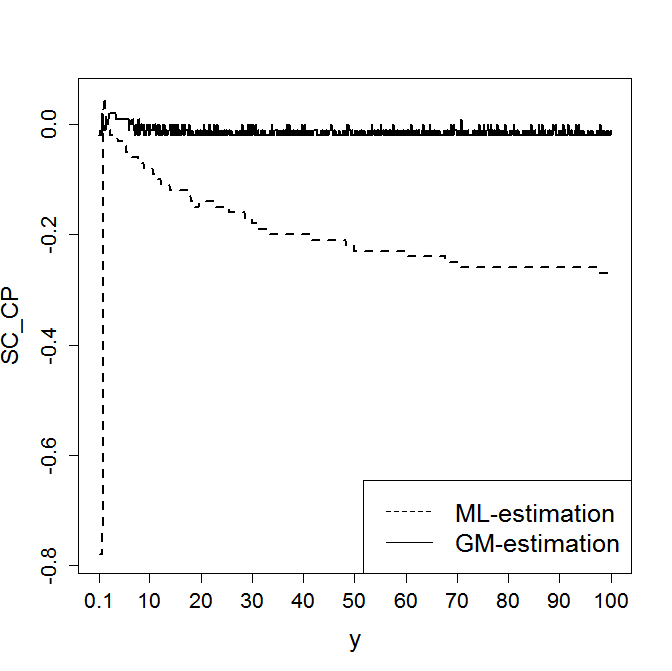}
		\caption{Coverage probability for contamination by one observation $y$ in a sample of size $n=100$}
	\label{fig:SenCurve}
\end{figure}

Examining the robustness for data which are contaminated by a value $y$ we can see that for the $ML$-estimator in all three dependence cases the coverage probability flattens for increasing $y$ but does not reach a constant value. This indicates a non-robust behaviour. The opposite can be seen for the $GM$-estimator, which coverage probability becomes constant when $y$ exceeds 5 and only fluctuates between two values. The behaviour of both estimators close to zero is similar. When $y$ decreases towards the lower bound of the distribution, both estimators have large deviations between the contaminated coverage probability and the uncontaminated one. Nevertheless the results concerning the robustness of the $GM$-estimator with $m=2$ are confirmed by the simulations. The results for $m=3,4,5$ were very similar, showing also a robust behaviour of the estimator by a constant coverage probability, and are therefore omitted here.

Altogether we can say that the $GM$-estimator is a good alternative to the $ML$-estimator and has similar coverage probability as well as length of the confidence interval even for small choices of $m$. These small choices give us an estimator, which is easy to calculate and for which we have shown that it is robust in contrast to the $ML$-estimator. This is underlined by the results in Figure \ref{fig:SenCurve}.

\section{Preliminary Results}
\label{sec:6}

In this section we state some results, which will help us to prove or main results.

First of all we want to use the (extended) variation condition not only for the kernel $h$, but also for the kernels $g_k$, $1\leq k\leq m$, of the Hoeffding decomposition. For that the following lemma is helpful.

\begin{lemma}\label{varcon}
$\newline$
If the kernel $h$ satisfies the extended variation condition, then the kernels $g_k$, $1\leq k \leq m$, satisfy it as well.
\end{lemma}

\begin{proof}
$\newline$
The proof will be made by mathematical induction.
Initially let $k=1$. We had defined $g_1$ as $g_1(x_1)=\mathbb{E}(h(x_1,Y_2,\ldots,Y_m))-\theta$. It is
\begin{align*}
&\mathbb{E}\left(\sup\limits_{\lVert(x_1,\ldots,x_m)-(X_1',\ldots,X_m')\rVert\leq \epsilon}\left| g_1(x_1)-g_1(X_1')\right|\right)\\
\leq&\mathbb{E}\left(\sup\limits_{\lVert x_1-X_1'\rVert\leq \epsilon} \mathbb{E}\left|(h(y_1,Y_2,\ldots,Y_m)-h(y_1',Y_2,\ldots,Y_m))\vert y_1=x_1, y_1'=X_1'\right|\right)\\
\leq&\mathbb{E}\left(\sup\limits_{\lVert x_1-X_1'\rVert\leq \epsilon} \left| h(x_1,Y_2,\ldots,Y_m)-h(X_1',Y_2,\ldots,Y_m)\right|\right)\\
\leq& L\epsilon,
\end{align*}
because $h$ satisfies the variation condition. So $g_1$ satisfies the extended variation condition.

$\newline$
Now let $g_{k-1}$ satisfy the extended variation condition. We show that $g_k$ also satisfies it:
\begin{align*}
g_k(x_1,\ldots,x_k)=&\mathbb{E}(h(x_1,\ldots,x_k,Y_{k+1},\ldots,Y_m))-\theta\\
&-\sum_{i=1}^k{g_1(x_i)}-\ldots -\sum_{1\leq i_1 \leq \ldots \leq i_{k-1}\leq k}{g_{k-1}(x_{i_1},\ldots,x_{i_{k-1}})}.
\end{align*}
The space of the functions satisfying the (extended) variation condition is a vector space (cf. \cite{Wen2011}) and since we know that all kernels up to $g_{k-1}$ satisfy the variation condition, it is sufficient to show that

 $\mathbb{E}(h(x_1,\ldots,x_k,Y_{k+1},\ldots,Y_m))-\theta$ satisfies the extended variation condition. 

\begin{align*}
&\mathbb{E}\left(\sup\limits_{\lvert x_1-Y_1\rvert\leq \delta}\Biggl| \mathbb{E}(h(y_1,X_2,\ldots,X_k,Y_{k+1},\ldots,Y_m)\vert y_1=x_1)\right.\\
&\left.\phantom{\mathbb{E}\left(\right)\mathbb{E}(h(y_1,X_2,\ldots,X_k))}-\mathbb{E}(h(y_1',X_2,\ldots,X_k,Y_{k+1},\ldots,Y_m)\vert y_1'=Y_1)\Biggr|\right)\\
&\leq \mathbb{E}\left(\sup\limits_{\lvert x_1-Y_1\rvert\leq \delta}\Biggl| h(x_1,X_2,\ldots,X_k,Y_{k+1},\ldots,Y_m)\right.\\
&\left.\phantom{\mathbb{E}\left(\right)\mathbb{E}(h(y_1,X_2,\ldots,X_k))}-h(Y_1,X_2,\ldots,X_k,Y_{k+1},\ldots,Y_m)\Biggr|\right)\\
&\leq L'\delta,
\end{align*}
since $h$ satisfies the extended variation condition. 
\end{proof}

\vspace*{0.5cm}

\begin{remark}
All results shown before for the extended variation condition without parameter $t$ remain true for the extended uniform variation condition.
\end{remark}

To ultimately show the asymptotic normality of $U$-statistics of strongly mixing random variables, we will first generalize some lemmas proved by \cite{Wen2011} respectively \cite{Deh2010} or \cite{Wen2011a} from the case $m=2$ to arbitrary $m$.

\vspace*{0.25cm}

First we need a covariance inequality, which we can establish by the coupling technique. A similar result for absolutely regular variables can be found in \cite{Yosh1976}.
Here we will follow \cite{Wen2011} and expand the lemma to the case $m\geq 2$, meaning we will treat $g_k$ for $2\leq k\leq m$. The proof is analogous to \cite{Wen2011} using the extended variation condition instead of the ordinary one and is therefore omitted.

\begin{lemma}\label{covineq}
$\newline$
Let $(X_n)_{n\in\mathbb{N}}$ be a strong mixing sequence of random variables with $\mathbb{E}\lvert X_1\rvert^{\rho}<\infty$ for a $\rho>0$ and $h$ a bounded kernel, which satisfies the extended variation condition. Moreover set $l=\max \lbrace i_{(2)}-i_{(1)}, i_{(2k)}-i_{(2k-1)}\rbrace$, where $\lbrace i_1,\ldots,i_{2k}\rbrace=\lbrace i_{(1)},\ldots,i_{(2k)} \rbrace$ and $i_{(1)}\leq \ldots \leq i_{(2k)}$. Then there exists a constant $C$, such that for all $2\leq k\leq m$
\begin{align*}
\lvert\mathbb{E}\left(g_k(X_{i_1},\ldots, X_{i_k})g_k(X_{i_{k+1}},\ldots,X_{i_{2k}})\right)\rvert\leq C \alpha^{\frac{\rho}{2\rho+1}}(l).
\end{align*}
\end{lemma}

\begin{lemma}\label{finalcov}
$\newline$
Let the kernel $h$ be bounded and satisfy the extended variation condition.
Let $(X_n)_{n\in\mathbb{N}}$ be a sequence of strong mixing random variables with $\mathbb{E}\left|X_1\right|^{\rho}<\infty$ for a $\rho>0$ and let $\sum_{l=0}^nl\alpha^{\frac{\rho}{2\rho+1}}(l)=O(n^{\gamma})$  for a $\gamma\geq 0$ hold.
Then for all $2\leq k\leq m$
\begin{align*}
\sum_{i_1,\ldots,i_{2k}=1}^n\left|\mathbb{E}(g_k(X_{i_1},\ldots,X_{i_k})g_k(X_{i_{k+1}},\ldots,X_{i_{2k}}))\right|=O(n^{2k-2+\gamma}).
\end{align*}
\end{lemma}

\begin{proof}
$\newline$
Set $\lbrace i_1,\ldots,i_{2k}\rbrace=\lbrace i_{(1)},\ldots,i_{(2k)}\rbrace$ with $i_{(1)}\leq \ldots \leq i_{(2k)}$. We can rewrite the above sum as
\begin{align*}
&\sum_{i_1,\ldots,i_{2k}=1}^n\left|\mathbb{E}(g_k(X_{i_1},\ldots,X_{i_k})g_k(X_{i_{k+1}},\ldots,X_{i_{2k}}))\right|\\
&=\sum_{l=0}^n\sum_{\stackrel{i_1,\ldots,i_{2k}=1}{\max\left\{ i_{(2)}-i_{(1)},i_{(2k)}-i_{(2k-1)}\right\}=l}}^n\left|\mathbb{E}(g_k(X_{i_1},\ldots,X_{i_k})g_k(X_{i_{k+1}},\ldots,X_{i_{2k}}))\right|\\
&\leq C \sum_{l=0}^n\sum_{\stackrel{i_1,\ldots,i_{2k},}{\max\lbrace i_{(2)}-i_{(1)},i_{(2k)}-i_{(2k-1)}\rbrace=l}} \alpha^{\frac{\rho}{1+2\rho}}(l),
\end{align*}
by application of Lemma \ref{covineq}.

For a further simplification we calculate via combinatorical arguments the quantity of the summands of the inner sum, that is the quantity of tuples $(i_1,\ldots,i_{2k})$ where $\max\lbrace i_{(2)}-i_{(1)},i_{(2k)}-i_{(2k-1)}\rbrace=l$. At first there are $(2k)!$ possibilities for a $2k$-tuple to get the same ordered sequence $i_{(1)},\ldots,i_{(2k)}$. Now we choose $i_{(1)}$ and $i_{(2k)}$ fixed and have $n^2$ possibilities for doing so. Through the requirement $\max\lbrace i_{(2)}-i_{(1)},i_{(2k)}-i_{(2k-1)}\rbrace=l$ we can also calculate the remaining possibilities for $i_{(2)}$ and $i_{(2k-1)}$. Suppose $i_{(2)}-i_{(1)}=\max\lbrace i_{(2)}-i_{(1)},i_{(2k)}-i_{(2k-1)}\rbrace=l$ then $i_{(2)}$ is automatically determined by the established choice of $i_{(1)}$. Because the requirement on the maximum still has to be fulfilled, $i_{(2k-1)}$ can only take $l$ distinct values. In the other case $i_{(2k)}-i_{(2k-1)}=\max\lbrace i_{(2)}-i_{(1)},i_{(2k)}-i_{(2k-1)}\rbrace=l$ we come to the same result. All remaining values of the $k$-tuple are arbitrary. Consequently the inner sum altogether is $(2k)!\cdot n^2ln^{2k-4}=l\cdot(2k)!\cdot n^{2k-2}$ and therefore

\begin{align*}
&\sum_{i_1,\ldots,i_{2k}=1}^n\left|\mathbb{E}(g_k(X_{i_1},\ldots,X_{i_k})g_k(X_{i_{k+1}},\ldots,X_{i_{2k}}))\right|\\
&\leq C'n^{2k-2}\sum_{l=0}^nl\alpha^{\frac{\rho}{1+2\rho}}(l)=O(n^{2k-2+\gamma}). 
\end{align*} 
\end{proof}

\vspace*{0.3cm}

We also need results concerning the remaining terms of the Hoeffding decomposition for $U$-processes. In this case we of course do not need simple convergence against zero, but since we consider processes need to have convergence of the supremum.

The following lemma was proved by \cite{Wen2011} for the case $m=2$. We will modify the main idea of the proof to obtain a similar result for the degenerated terms of higher dimensional $U$-processes.

\begin{lemma}\label{remterm}
$\newline$
Let $h$ be a kernel satisfying the extended uniform variation condition, such that the $U$-distribution function $U$ is Lipschitz-continuous. Moreover let 
 $(X_n)_{n\in\mathbb{N}}$ be a sequence of strong mixing random variables with mixing coefficients $\alpha(l)=O(l^{-\delta})$ for $\delta\geq 8$ and $\mathbb{E}\lvert X_i\rvert^{\rho}<\infty$ for a $\rho>\frac{1}{4}$.
 Then for all $2\leq k \leq m$ and $\gamma=\frac{\delta-2}{\delta}$ we have 
\begin{align*}
\sup\limits_{t\in\mathbb{R}}\big\vert\sum_{1\leq i_1,\ldots,i_k\leq n}g_k(X_{i_1},\ldots,X_{i_k},t)\big\vert=o(n^{k-\frac{1}{2}-\frac{\gamma}{8}})~\text{a.s.}.
\end{align*}
\end{lemma}

\begin{proof}
$\newline$
We define $Q_n^k(t):=\sum_{1\leq i_1,\ldots,i_k\leq n}h_k(X_{i_1},\ldots,X_{i_k},t)$. 

For $l \in \mathbb{N}$ choose $t_{1,l},\ldots,t_{s-1,l}$ with $s=s_l=O(2^{\frac{5}{8}l})$, such that 
\begin{align*}
-\infty=t_{0,l}<t_{1,l}<\ldots<t_{s-1,l}<t_{s,l}=\infty
\end{align*}
and $2^{-\frac{5}{8}l}\leq \lvert U(t_{r,l}-U(t_{r-1,l})\rvert\leq 2\cdot 2^{-\frac{5}{8}l}$. Since we required Lipschitz-continuity of $U$ it follows that $2^{-\frac{5}{8}l}\leq C\lvert t_{r,l}-t_{r-1,l}\rvert$. Moreover, because $h$ is non-decreasing in $t$, 

$\mathbb{E}\left(h(Y_{1},\ldots,Y_{k},Y_{k+1},\ldots Y_m,t)\vert Y_1=X_{i_1},\ldots,Y_k=X_{i_k}\right)$ is non-decreasing in $t$ for all $2\leq k \leq m$. We proceed by induction.

\vspace*{0,1cm}
The case $k=2$ was treated by \cite{Wen2011} and is therefore omitted here.

\vspace*{0.25cm}
From now on suppose that the statement of the lemma is valid for $k-1$. 

Together with the above consideration we have for every $t\in [t_{r-1,l},t_{r,l}]$ and $2^l\leq n < 2^{l+1}$
\begin{align*}
&\lvert Q_n^k(t)\rvert \\
=&\big\vert \sum_{1\leq i_1<\ldots <i_k\leq n}\big(\mathbb{E}(h(Y_1,\ldots Y_k,Y_{k+1},\ldots,Y_m,t)\vert Y_1=X_{i_1},\ldots,Y_k=X_{i_k})\\
&\phantom{\big\vert \sum_{1\leq i_1<\ldots <i_k\leq n}\big(\big)}-g_1(X_{i_1},t)-\cdots-g_1(X_{i_k},t)\\
&\phantom{\big\vert \sum_{1\leq i_1<\ldots <i_k\leq n}\big(\big)}-g_2(X_{i_1},X_{i_2},t)-\cdots-g_2(X_{i_{k-1}},X_{i_k},t)-\cdots-U(t)\big)\big\vert\\
\leq& \max\bigg\{\big\vert \sum_{1\leq i_1<\ldots<i_k\leq n}\big(\mathbb{E}(h(X_{i_1},\ldots,X_{i_k},Y_{i_{k+1}},\ldots,Y_m,t_{r,l})\\
&\phantom{\big\vert \sum_{1\leq i_1<\ldots <i_k\leq n}\big(\big)}-g_1(X_{i_1},t_{r,l})-\ldots -g_1(X_{i_k},t_{r,l})\\
&\phantom{\sum_{1\leq i_1<\ldots <i_k\leq n}}-g_2(X_{i_1},)-\cdots-g_2(X_{i_{k-1}},X_{i_k},t_{r,l})-\cdots-U(t_{r,l}))\big)\big\vert,\\
&\phantom{\max\bigg\{}\big\vert \sum_{1\leq i_1<\ldots <i_k\leq n}\big(\mathbb{E}(h(X_{i_1},\ldots,X_{i_k},Y_{i_{k+1}},\ldots,Y_m,t_{r-1,l}))\\
&\phantom{\big\vert \sum_{1\leq i_1<\ldots <i_k\leq n}\big(\big)}-g_1(X_{i_1},t_{r-1,l})-\ldots -g_1(X_{i_k},t_{r-1,l})\\
&\phantom{\sum_{1\leq i_1<\ldots <i_k\leq n}}-g_2(X_{i_1},X_{i_2},t_{r-1,l})-\cdots-g_2(X_{i_{k-1}},X_{i_k},t_{r-1,l})\\
&\phantom{\big\vert \sum_{1\leq i_1<\ldots <i_k\leq n}\big(hallowelt\big)}-\cdots-U(t_{r-1,l})\big)\big\vert\bigg\}\\
&+\binom{n-1}{k-1}\max\left\{\left\vert\sum_{i=1}^n(g_1(X_i,t_{r,l})-g_1(X_i,t))\right\vert,\right.\\
&\phantom{+\binom{n-2}{k-2}\max\sum_{i=1}^n(g_2(X_{i_1},))}\left. \left\vert\sum_{i=1}^n(g_1(X_i,t)-g_1(X_i,t_{r-1,l}))\right\vert\right\}\\
&+\binom{n-2}{k-2}\max\left\{\left\vert\sum_{i=1}^n(g_2(X_{i_1},)-g_2(X_{i_1},X_{i_2},t))\right\vert,\right.\\
&\left.\phantom{+\binom{n-2}{k-2}\max\sum_{i=1}^n(g_2(X_{i_1},))} \left\vert\sum_{i=1}^n(g_2(X_{i_1},X_{i_2},t)-g_2(X_{i_1},X_{i_2},t_{r-1,l}))\right\vert\right\}\\
&+\cdots+\binom{n}{k}\lvert U(t_{r,l})-U(t_{r-1,l})\rvert\\
\leq &  \max\lbrace \lvert Q_n^k(t_{r,l})\rvert,\lvert Q_n^k(t_{r-1,l})\rvert\rbrace\\
&+\binom{n-1}{k-1}\left\vert\sum_{i=1}^n(g_1(X_i,t_{r,l})-g_1(X_i,t_{r-1,l}))\right\vert\\
&+\binom{n-2}{k-2}\left\vert \sum_{1\leq i_1<i_2\leq n} \left(g_2(X_{i_1},X_{i_2},t_{r,l})-g_2(X_{i_1},X_{i_2},t_{r-1,l})\right)\right\vert\\
&+\cdots+\binom{n-(k-1)}{k-(k-1)}\big\vert\sum_{1\leq i_1<\ldots<i_{k-1}\leq n} \left(g_{k-1}(X_{i_1},\ldots,X_{i_{k-1}},t_{r,l})\right.\\
&\phantom{+\binom{n-(k-1)}{k-(k-1)}\left\vert\sum_{1\leq i_1<\ldots<i_{k-1}\leq n} \left(g_{k-1}\right)\right\vert}\left.-g_{k-1}(X_{i_1},\ldots,X_{i_{k-1}},t_{r-1,l})\right)\big\vert\\
&+\binom{n}{k}\lvert U(t_{r,l})-U(t_{r-1,l})\rvert.
\end{align*}

Again we will treat the first, second and last summand separately.

For the first summand follows
\begin{align*}
&\mathbb{E}\left(\max\limits_{n=2^l,\ldots,2^{l+1}-1}\max\limits_{r=0,\ldots,s}\lvert Q_n^k(t_{r,l})\rvert^2\right)\\
&\leq \sum_{r=0}^{s}\mathbb{E}\left(\left(\sum_{d=0}^l\max\limits_{i=1,\ldots,2^{l-d}}\lvert Q_{2^l+i2^d}^k(t_{r,l})-Q_{2^l+(i-1)2^d}^k(t_{r,l})\rvert\right)^2\right)\\
&\leq \sum_{r=0}^sl\sum_{d=0}^l\sum_{i=1}^{2^{-d}}\mathbb{E}\left(\left(Q_{2^l+i2^d}^k(t_{r,l})-Q_{2^l+(i-1)2^d}^k(t_{r,l})\right)^2\right)\\
&\leq \sum_{r=0}^sl\sum_{d=0}^l\underbrace{\sum_{i_1,\ldots,i_{4}=1}^{2^{l+1}}\lvert \mathbb{E}\left(g_k(X_{i_1},\ldots,X_{i_k},t)g_k(X_{i_{k+1}},\ldots,X_{i_{2k}},t)\right)\rvert}_{=O((2^{l+1})^{2k-2+\gamma})\text{, with Lemma \ref{finalcov}}}\\
&\leq sl^2 C 2^{(2k-2)(l+1)}\leq C'l^22^{(2k-2+\frac{5}{8})l}.
\end{align*}

For the first inequality we used the so called chaining technique: via the triangular inequality we parted the term $Q_n$ into two differences $Q_{2^l+i2^d}-Q_{2^l+(i-1)2^d}$.

Now we apply the Chebychev inequality getting for every $\epsilon>0$
\begin{align*}
&\sum_{l=1}^\infty\mathbb{P}\left(\max\limits_{n=2^l,\ldots,2^{l+1}-1}\max\limits_{r=0\ldots,s}\lvert Q_n^k(t_{r,l})\rvert>\epsilon2^{l(k-\frac{1}{2}-\frac{\gamma}{8})}\right)\\
&\leq\sum_{l=1}^\infty\frac{1}{\epsilon^22^{l(2k-1-\frac{\gamma}{4})}}\mathbb{E}\left(\max\limits_{n=2^l,\ldots,2^{l+1}-1}\max\limits_{r=0\ldots,s}\lvert Q_n^k(t_{r,l})\rvert^2\right)\\
&\leq \sum_{l=1}^\infty \frac{1}{\epsilon^22^{l(2k-1-\frac{\gamma}{4})}}C'l^22^{(2k-2+\frac{5}{8})l}\leq \sum_{l=1}^\infty C'\frac{l^2}{\epsilon^2}2^{\frac{-3+2\gamma}{8}l}<\infty.
\end{align*}
Then with the Borel-Cantelli Lemma
\begin{align*}
\mathbb{P}\left(\max\limits_{n=2^l,\ldots,2^{l+1}-1}\max\limits_{r=0,\ldots,s}\lvert Q_n^2(t_{r,l})\rvert>\epsilon2^{l(k-\frac{1}{2}-\frac{\gamma}{8})} \text{ infinitely often}\right)=0.
\end{align*}
That is, $\max\limits_{r=0,\ldots,s}\lvert Q_n^k(t_{r,l})\rvert=o(n^{k-\frac{1}{2}-\frac{\gamma}{8}})$.

Now we will treat the second summand for which we want to apply Lemma 4.2.1 of \cite{Wen2011}. For $2^l\leq n <2^{l+1}$ it follows
\begin{align*}
&\mathbb{E}\left(\sum_{i=1}^n\left(g_1(X_i,t_{r,l})-g_1(X_i,t_{r-1,l})\right)\right)^4\\
&\leq Cn^2(\log n)^2\max\left\{\mathbb{E}\lvert g_1(X_i,t_{r,l})-g_1(X_i,t_{r-1,l})\rvert, Cn^{-\frac{3}{4}}\right\}^{1+\gamma}\\
&\leq Cn^2(\log n)^2(Cn^{-\frac{3}{4}})^{1+\gamma}.
\end{align*}
By usage of the assumption  $\lvert U(t_{r,l})-U(t_{r-1,l})\rvert \geq 2^{-\frac{5}{8}l} \geq C2^{-\frac{3}{4}l}\geq Cn^{-\frac{3}{4}}$, the last term simplifies to
\begin{align*}
Cn^2(\log n)^2\lvert U(t_{r,l})-U(t_{r-1,l})\rvert^{1+\gamma}.
\end{align*}
All in all we get
\begin{align*}
&\mathbb{E}\left(\max\limits_{n=2^l,\ldots,2^{l+1}-1}\max\limits_{r=1\ldots,s}\binom{n-1}{k-1}\lvert\sum_{i=1}^n\left(g_1(X_i,t_{r,l})-g_1(X_i,t_{r-1,l})\right)\rvert\right)^4\\
&\leq n^{4(k-1)} \sum_{r=1}^s \mathbb{E}\left(\max\limits_{n=2^l,\ldots,2^{l+1}-1}\lvert\sum_{i=1}^n\left(g_1(X_i,t_{r,l})-g_1(X_i,t_{r-1,l})\right)\rvert\right)^4\\
&\leq n^{4(k-1)} \sum_{r=1}^s Cn^2(\log n)^2\lvert U(t_{r,l})-U(t_{r-1,l})\rvert^{1+\gamma}\\
&\leq 2^{4(k-1)(l+1)}Cn^2(\log n)^2s\left(\max\limits_{r=1\ldots,s}\lvert U(t_{r,l})-U(t_{r-1,l})\rvert\right)^{1+\gamma}\\
&\leq C'(l+1)^22^{(4k-2-\frac{5}{8}\gamma)l}.
\end{align*}
Thereby we used Corollary 1 of \cite{Mor1983} and the assumption $s=O(2^{\frac{5}{8}l})$. 

Analogously to the above calculation we again apply the generalized Chebychev Inequality and the Borel-Cantelli Lemma getting 

\begin{align*}
\binom{n-1}{k-1}\left\vert\sum_{i=1}^n\left(g_1(X_i,t_{r,l})-g_1(X_i,t_{r-1,l})\right)\right\vert=o(n^{k-\frac{1}{2}-\frac{1}{8}\gamma}).
\end{align*}
For the last summand, using the assumptions and the fact that  $\gamma<1$, we have
\begin{align*}
&\max\limits_{r=0,\ldots,s}\binom{n}{k}\lvert U(t_{r,l})-U(t_{r-1,l})\rvert\leq Cn^k2^{-\frac{5}{8}l}\leq Cn^{k-\frac{5}{8}}<Cn^{k-\frac{4}{8}-\frac{1}{8}\gamma}\\
&\phantom{\max\limits_{r=0,\ldots,s}\binom{n}{k}\lvert U(t_{r,l})-U(t_{r-1,l})\rvert\leq Cn^k2^{-\frac{5}{8}l}}=o(n^{k-\frac{4}{8}-\frac{1}{8}\gamma}).
\end{align*}
\vspace*{0.25cm}

Now the terms including $g_2,\ldots,g_{k-1}$ remain. For these we know for $2\leq j\leq k-1$
 
\begin{align*}
\sup\limits_{t\in\mathbb{R}}\bigg\vert\sum_{1\leq i_1,\ldots,i_j\leq n}g_j(X_{i_1},\ldots,X_{i_j},t)\bigg\vert=o(n^{j-\frac{1}{2}-\frac{\delta-2}{8\delta}})
\end{align*}
 and consequently
\begin{align*}
 &\binom{n-j}{k-j}\max\limits_{r=1\ldots,s}\lvert \sum_{1\leq i_1<\ldots<i_j\leq n} \left(g_j(X_{i_1},\ldots,X_{i_j},t_{r,l})-g_j(X_{i_1},\ldots,X_{i_j},t_{r-1,l})\right)\rvert\\
 &\leq n^{k-j}\left(\max\limits_{r=1\ldots,s}\lvert \sum_{1\leq i_1<\ldots<i_j\leq n}g_j(X_{i_1},\ldots,X_{i_j},t_{r,l})\rvert\right.\\
& \phantom{\leq n^{k-j}\left(\right)\max\limits_{r=}\lvert \sum_{1\leq i_1<\ldots<i_j\leq n}\rvert}+ \left.\max\limits_{r=1\ldots,s}\lvert \sum_{1\leq i_1<\ldots<i_j\leq n}g_j(X_{i_1},\ldots,X_{i_j},t_{r-1,l})\rvert\right)\\
 &\leq n^{k-j} o(n^{j-\frac{1}{2}-\frac{1}{8}\frac{\delta-2}{\delta}})= o(n^{k-\frac{1}{2}-\frac{1}{8}\frac{\delta-2}{\delta}}).
\end{align*}

So we could show for arbitrary $k$ and all sumands that they are of order $o(n^{k-\frac{1}{2}-\frac{1}{8}\frac{\delta-2}{\delta}})$. Using mathematical induction the proof is completed. 
\end{proof}

\section{Proofs}
\label{sec:7}

In this section we give the missing proofs of the main results stated in Section \ref{sec:2}.

\vspace*{0.1cm}

\begin{proof}[Theorem \ref{main}]

For the main proof we have to show that the following three conditions are fulfilled. \cite{Serf1984} has already proved that these conditions together are sufficient to show asymptotic normality. From there one can see that independence is not required, if these conditions are fulfilled. Some of the lemmas used for proving this theorem can also be found in \cite{Cho1988}.

\begin{compactenum}[(i)]
\item 
For $W_{H_n,H_F}(y)=\left(\frac{\int_0^{H_n(y)}{J(t)dt}-\int_0^{H_F(y)}{J(t)dt}}{H_n(y)-H_F(y)}-J(H_F(y))\right)$ holds

$\lVert W_{H_n,H_F} \rVert_{L_1}=o_p(1)$ and it is $\lVert H_n-H_F \rVert_{\infty}=O_p(n^{-\frac{1}{2}})$.
\item For the remainder term $R_{p_i,n}=\hat{\xi}_{p_i,n}-\xi_{p_i}+\frac{p_i-H_n(\xi_{p_i})}{h_f(\xi_{p_i})}$ of the Bahadur representation of an empirical quantile holds 
\begin{align*}
R_{p_i,n}=o_p(n^{-\frac{1}{2}}).
\end{align*}
\item For a $U$-statistic with kernel 
\begin{align*}
A(x_1,\ldots,x_m)=&-\int_{-\infty}^{\infty}{\left(1_{\left[h(x_1,\ldots,x_m)\leq y\right]}-H_F(y)\right)J(H_F(y))dy}\\
&+\sum_{i=1}^{d}{a_i\frac{p_i-1_{\left[h(x_1,\ldots,x_m)\leq H_F^{-1}(p_i)\right]}}{h_F(H_F^{-1}(p_i))}}
\end{align*}
 we have
\begin{align*}
\sqrt{n}(U_n(A)-\theta)\stackrel{D}{\longrightarrow}N(0,\sigma^2).
\end{align*}
\end{compactenum}

\textit{Proofs of the conditions}

Now we show that the conditions (i)-(iii) are satisfied.
\vspace*{0.15cm}

 For the first part of condition (i) we refer to Lemma 8.2.4.A of \cite{serf1980}. Although he demands independence of the random variables in his proof this property is not needed.
The second part of condition (i) follows from Corollary \ref{glican} .

\vspace*{0.15cm}

Condition (ii) is fulfilled by Lemma \ref{bahad}.

\vspace*{0.15cm}

It remains to show that condition (iii) is satisfied. 

For this we apply Theorem \ref{asynom}. We merely have to verify, whether $A$ satisfies the assumptions for the kernel, that is (a) $A$ is bounded and (b) satisfies the extended variation condition.
\vspace*{0.25cm}
We consider again the kernel $A$
\begin{align*}
A(x_1,\ldots,x_m)=&-\int_{-\infty}^{\infty}\left(1_{[h(x_1,\ldots,x_m)\leq y]}-H_F(y)\right)J(H_F(y))dy\\
&+\sum_{i=1}^d a_i\frac{p_i-1_{[h(x_1,\ldots,x_m)\leq H_F^{-1}(p_i)]}}{h_F(H_F^{-1}(p_i))}.
\end{align*}

\vspace*{0.5cm}

\begin{enumerate}[(a)]
\item The boundedness is a result of the continuity of $H_F$ and $J$ and that $J$ vanishes off the interval $[\alpha,\beta]$. 
\item Now we want to show that $A$ satisfies the extended variation condition. We will treat both summands separately, at first for arbitary $y_1,\ldots,y_m$:
\begin{align*}
&\mathbb{E}\left( \sup\limits_{\lVert(x_1,\ldots,x_m)-({y}_1,\ldots,{y}_m)\rVert\leq \epsilon }\lvert A(x_1,\ldots,x_m)-A({y}_1,\ldots,{y}_m)\rvert\right)\\
\leq & \mathbb{E}\bigg( \sup\limits_{\lVert(x_1,\ldots,x_m)-({y}_1,\ldots,{y}_m)\rVert\leq \epsilon }\bigg\vert \int_{-\infty}^{\infty}\left(1_{[h(x_1,\ldots,x_m)\leq y]}-H_F(y)\right)J(H_F(y))dy\\
&~-\int_{-\infty}^{\infty}\left(1_{[h({y}_1,\ldots,{y}_m)\leq y]}-H_F(y)\right)J(H_F(y))dy \bigg\vert\bigg)\\
+& \mathbb{E}\bigg( \sup\limits_{\lVert(x_1,\ldots,x_m)-({y}_1,\ldots,{y}_m)\rVert\leq \epsilon }\bigg\vert
\sum_{i=1}^d a_i\frac{p_i-1_{[h(x_1,\ldots,x_m)\leq H_F^{-1}(p_i)]}}{h_F(H_F^{-1}(p_i))}\\
&~-\sum_{i=1}^d a_i\frac{p_i-1_{[h({y}_1,\ldots,{y}_m)\leq H_F^{-1}(p_i)]}}{h_F(H_F^{-1}(p_i))} \bigg\vert\bigg)\\
\leq & \mathbb{E}\Biggl(\sup\limits_{\lVert(x_1,\ldots,x_m)-({y}_1,\ldots,{y}_m)\rVert\leq \epsilon }\big\vert \int_{-\infty}^{\infty}\bigl(1_{[h(x_1,\ldots,x_m)\leq y]}\\
&\phantom{\mathbb{E}\big( }\underbrace{\phantom{\sup\limits_{\lVert(x_1,\ldots,x_m)-({y}_1,\ldots,{y}_m)\rVert}\big)}-1_{[h({y}_1,\ldots,{y}_m)\leq y]}\bigr)J(H_F(y))dy \big\vert}_{=:A_1(y_1,\ldots,y_m)}\Biggr)\\
+&\mathbb{E}\Biggl( \sup\limits_{\lVert(x_1,\ldots,x_m)-({y}_1,\ldots,{y}_m)\rVert\leq \epsilon }\\
&\phantom{\mathbb{E}\big( }\underbrace{\phantom{\sup}\left\vert
\sum_{i=1}^d a_i\frac{p_i-1_{[h({y}_1,\ldots,{y}_m)\leq H_F^{-1}(p_i)]}-1_{[h(x_1,\ldots,x_m)\leq H_F^{-1}(p_i)]}}{h_F(H_F^{-1}(p_i))} \right\vert}_{=:A_2(y_1,\ldots,y_m)}\Biggr).
\end{align*}

For the verification of the simple variation condition we first treat $A_1$ getting
\begin{align*}
&A_1(X'_1,\ldots,X'_m)\\
&\leq\Bigg\vert\int_{-\infty}^{\infty}J(H_F(y))\\
&\phantom{\leq\left\vert\int_{-\infty}^{\infty}\right\vert}\sup\limits_{\lVert(x_1,\ldots,x_m)-({X'}_1,\ldots,{X'}_m)\rVert\leq \epsilon }\left\vert1_{[h({x}_1,\ldots,{x}_m)\leq y]}-1_{[h({X'}_1,\ldots,{X'}_m)\leq y]}\right\vert dy\Bigg\vert
\end{align*}

Using the Lipschitz-continuity we have
\begin{align*}
&\sup\limits_{\lVert(x_1,\ldots,x_m)-({X'}_1,\ldots,{X'}_m)\rVert\leq \epsilon }\left\vert1_{[h({X'}_1,\ldots,{X'}_m)\leq y]}-1_{[h({x}_1,\ldots,{x}_m)\leq y]}\right\vert\\
=&\begin{cases}
1 &\text{, if } h({X'}_1,\ldots,{X'}_m) \in \left(y-\widetilde{L}\epsilon,
 y+\widetilde{L}\epsilon\right)\\
0 &\text{, else}.
\end{cases}
\end{align*}
One can easily see that $C:=\left\vert\int_{-\infty}^{\infty}J(H_F(y))dy\right\vert$ is bounded.
Therefore 
\begin{align*}
&\mathbb{E}(A_1(X'_1,\ldots,X'_m))\\
&\leq \mathbb{E} \left(\sup\limits_{t\in\mathbb{R}}\left\vert 1_{[h({X'}_1,\ldots,{X'}_m)\in \left(t-\widetilde{L}\epsilon,
    t+\widetilde{L}\epsilon\right)]}\right\vert \cdot \left\vert\int_{-\infty}^{\infty}J(H_F(y))dy\right\vert\right)\\
&\leq\sup\limits_{t\in\mathbb{R}}\left\vert\mathbb{E}\left( 1_{[h({X'}_1,\ldots,{X'}_m)\in \left(t-\widetilde{L}\epsilon,
    t+\widetilde{L}\epsilon\right)]}\right)\right\vert \cdot C\\
&\leq C\cdot\sup\limits_{t\in\mathbb{R}} \left\vert\mathbb{P}\left(h({X'}_1,\ldots,{X'}_m)\in \left(t-\widetilde{L}\epsilon,
 t+\widetilde{L}\epsilon\right)\right)\right\vert\\
 &\leq C\cdot\sup\limits_{t\in\mathbb{R}}\left\vert\int_{t-\widetilde{L}\epsilon}^{t+\widetilde{L}\epsilon}h_F(x)dx\right\vert\leq C\cdot\left(\sup\limits_{x\in\mathbb{R}}h_F(x)\right)2\widetilde{L}\epsilon\leq L\epsilon,
\end{align*}
 since $h_F$ is bounded.
 
\vspace*{0.25cm}

The treatment of $A_2$ is analogous, using the same notation of the supremum as above. Therefore $A$ satisfies the variation condition and using the same arguments for the extended variation condition the proof is finished.
\end{enumerate}

\vspace*{0.15cm}

We have shown conditions (i)-(iii) and so the proof of asymptotic normality is completed. 
\end{proof}

\vspace*{0.2cm}

\begin{proof}[Theorem \ref{bahad}]
$\newline$
Let be $t\in \mathbb{R}$, $\xi_{nt}=\xi_p+tn^{-\frac{1}{2}}, Z_n(t)=\sqrt{n}\frac{H_F(\xi_{nt})-H_n(\xi_{nt})}{h_F(\xi_p)}$ and $V_n(t)=\sqrt{n}\frac{H_F(\xi_{nt})-H_n(\hat{\xi}_{p})}{h_F(\xi_p)}$.

Using $\lvert p-H_n(\hat{\xi}_p)\rvert \leq \tfrac{1}{n}$ we obtain 

\begin{align*}
V_n(t)&=\underbrace{\sqrt{n}\frac{H_F(\xi_p+tn^{-\frac{1}{2}})-p}{h_F(\xi_p)}}_{=:V'_n(t)}+\underbrace{\sqrt{n}\frac{\overbrace{p-H_n(\hat{\xi}_p)}^{=O(n^{-1})}}{h_F(\xi_p)}}_{=O(n^{-\frac{1}{2}})}\longrightarrow t.
\end{align*}

\vspace*{0.5cm}

Next we will show that $Z_n(t)-Z_n(0)\stackrel{P}{\longrightarrow}0$. One can easily see that

\begin{align*}
&Var(Z_n(t)-Z_n(0))\\
&= \frac{n}{h_F^2(\xi_p)}Var\Biggl(\frac{1}{\binom{n}{m}}\sum_{1\leq i_1<\ldots<i_m\leq n}1_{\left[h(X_{i_1},\ldots,X_{i_m})\leq \xi_p+tn^{-\frac{1}{2}}\right]}\\
&\phantom{\frac{n}{h_F^2(\xi_p)}Var\left(\frac{1}{\binom{n}{m}}\sum_{1\leq i_1<\ldots<i_m\leq n}1_{\left[h(X_{i_1},)\right]}\right)}-1_{\left[h(X_{i_1},\ldots,X_{i_m})\leq \xi_p\right]}\Biggr).\\
\end{align*}

 To find bounds for the right hand side, we define $U_n$ and $U_n'$ as 
\begin{align*}
U_n&=\frac{1}{\binom{n}{m}}\sum_{1\leq i_1<\ldots<i_m\leq n}1_{\left[h(X_{i_1},\ldots,X_{i_m})\leq \xi_p+tn^{-\frac{1}{2}}\right]}\\
&=\theta+\sum_{j=1}^{m}\binom{m}{j}\frac{1}{\binom{n}{j}}\sum_{1\leq i_1<\ldots<i_j\leq n}g_k(X_{i_1},\ldots,X_{i_k})
\end{align*}

\begin{align*}
U'_n&=\frac{1}{\binom{n}{m}}\sum_{1\leq i_1<\ldots<i_m\leq n}1_{\left[h(X_{i_1},\ldots,X_{i_m})\leq \xi_p\right]}\\
&=\theta'+\sum_{j=1}^{m}\binom{m}{j}\frac{1}{\binom{n}{j}}\sum_{1\leq i_1<\ldots<i_j\leq n}g'_k(X_{i_1},\ldots,X_{i_k}),
\end{align*}

where $g_k$ and $g'_k$ are the related terms of the Hoeffding decomposition as used before.

Therefore we have 
\begin{align*}
&\sqrt{Var\left(\frac{1}{\binom{n}{m}}\sum_{1\leq i_1<\ldots<i_m\leq n}1_{\left[\xi_p<h(X_{i_1},\ldots,X_{i_m})\leq \xi_p+tn^{-\frac{1}{2}}\right]}\right)}\\
\leq& \sqrt{\underbrace{Var(\theta)}_{=0}}+\sqrt{\underbrace{Var(\theta')}_{=0}}+\sqrt{Var\left(\frac{m}{n}\sum_{i=1}^n(g_1(X_i)-g'_1(X_i))\right)}\\
+&\sqrt{\var\left(\frac{\binom{m}{2}}{\binom{n}{2}}\sum_{1\leq i <j \leq n}g_2(X_i,X_j)\right)}+\sqrt{Var\left(\frac{\binom{m}{2}}{\binom{n}{2}}\sum_{1\leq i <j \leq n}g'_2(X_i,X_j)\right)}\\
+&\ldots+\sqrt{\var\left(\frac{1}{\binom{n}{m}}\sum_{1\leq i_1<\ldots<i_m\leq n}g_m(X_{i_1},\ldots,X_{i_m})\right)}\\
+&\sqrt{\var\left(\frac{1}{\binom{n}{m}}\sum_{1\leq i_1<\ldots<i_m\leq n}g'_m(X_{i_1},\ldots,X_{i_m})\right)}.\\
\end{align*}
We have shown in the proof of Theorem \ref{asynom} that for all $2\leq k \leq m$  it is $$Var\left(\frac{\binom{m}{k}}{\binom{n}{k}}\sum_{1\leq i_1<\ldots<i_k\leq n}g_k(X_{i_1},\ldots,X_{i_k})\right)=O(n^{-2+\gamma})$$
 for a $\gamma<1$, if the kernel is bounded and satisfies the extended variation condition. Analogous to the proof of Corollary \ref{glican} we know that $g(x_1,\ldots,x_m)=1_{\left[h(X_{i_1},\ldots,X_{i_m})\leq \xi_p+tn^{-\frac{1}{2}}\right]}$ and $g'(x_1,\ldots,x_m)=1_{\left[h(X_{i_1},\ldots,X_{i_m})\leq \xi_p\right]}$  satisfy the extended variation condition.

Applying Proposition 1 of \cite{Dou2010} on $g_1(X_i)-g'_1(X_i)$ and $p=2, b=3$ and using $\lVert g_1(X_i)-g'_1(X_i)\rVert_3< \infty$, since the kernels are bounded, we have
\begin{align*}
\mathbb{E}\left \vert\sum_{i=1}^n(g_1(X_i)-g'_1(X_i))\right \vert^2\leq Cn,
\end{align*}
where the constant $$C=4 \left(\int_0^1\left(\min\left\{\sum_{i\geq 0}1_{[u<\alpha(i)]},n\right\}\right)^3du\right)^{\frac{1}{3}}\lVert g_1(X_i)-g'_1(X_i) \rVert^2_3$$
 only depends on $\lVert g_1(X_i)-g'_1(X_i) \rVert_3$, since \cite{Douk2009} proved $$\left(\int_0^1\left(\min\left\{\sum_{i\geq 0}1_{[u<\alpha(i)]},n\right\}\right)^3du\right)^{\frac{1}{3}}<\infty.$$

So we get
\begin{align*}
&\sqrt{\var\left(\frac{1}{\binom{n}{m}}\sum_{1\leq i_1<\ldots<i_m\leq n}1_{\left[\xi_p<h(X_{i_1},\ldots,X_{i_m})\leq \xi_p+tn^{-\frac{1}{2}}\right]}\right)}\\
&\leq \sqrt{\frac{m^2}{n^2}Cn}+2(m-1)\sqrt{O(n^{-2+\gamma})}\leq \frac{Cm^2}{\sqrt{n}}+2(m-1)O(n^{-1+\gamma/2}),
\end{align*}
where the constant $C$ only depends on $\left\Vert g_1(X_i)- g'_1(X_i) \right\Vert_3$.

Let us come back to

\begin{align*}
&Var(Z_n(t)-Z_n(0))\\
&\leq \frac{n}{h^2_F(\xi_p)}\left(\frac{Cm^2}{\sqrt{n}}+2(m-1)O(n^{-1+\gamma/2})\right)^2\\
&\leq \frac{m^2}{h^2_F(\xi_p)} C^2 + \frac{4m^2(m-1)}{h^2_F(\xi_p)}C \sqrt{n}O(n^{-1+\gamma/2})+4(m-1)^2 O(n^{-2+\gamma})\\
&\leq \frac{m^2}{h^2_F(\xi_p)} C^2+ \frac{4m^2(m-1)}{h^2_F(\xi_p)}C O(n^{-\frac{1}{2}+\gamma/2})+4(m-1)^2 O(n^{-2+\gamma}).
\end{align*}

Since $\lvert g_1(X_i)-g'_1(X_i)\rvert\leq 1$ for all $X_i$ and
\begin{align*}
&\lvert g_1(X_i)-g'_1(X_i)\rvert\stackrel{P}{\longrightarrow} 0
\end{align*}
the constant $C$ converges to zero in probability and therefore

$$\var(Z_n(t)-Z_n(0))\stackrel{P}{\longrightarrow}0.$$

\vspace*{0.1cm}

Applying the Chebychev inequality we then have $Z_n(t)-Z_n(0)\stackrel{P}{\longrightarrow}0$ .

Altogether we have for $t\in\mathbb{R}$ and every $\epsilon>0$
\begin{align}
\mathbb{P}(\sqrt{n}(\hat{\xi}_p-\xi_p)&\leq t, Z_n(0)\geq t+\epsilon)=\mathbb{P}(Z_n(t)\leq V_n(t),Z_n(0)\geq t+ \epsilon)\nonumber\\
&\leq\mathbb{P}\left(\lvert Z_n(t)-Z_n(0)\rvert\geq\frac{\epsilon}{2}\right)+\mathbb{P}\left(\lvert V_n(t)-t\rvert\geq \frac{\epsilon}{2}\right)\longrightarrow 0, \nonumber
\end{align}

and analogously
\begin{align*}
\mathbb{P}\left(\sqrt{n}(\hat{\xi}_p-\xi_p)\geq t, Z_n(0)\leq t\right)\longrightarrow 0.
\end{align*}

Using Lemma 1 of \cite{Gho1971} the proof is completed. 
\end{proof}

\vspace*{0.2cm}

\begin{proof}[Theorem \ref{asynom}]
$\newline$
The proof makes use of the Hoeffding decomposition
\begin{align*}
\sqrt{n}(U_n-\theta)=\sqrt{n}\sum_{j=1}^m{\binom{m}{j}\frac{1}{\binom{n}{j}}S_{jn}}.
\end{align*}
We show that the linear part $\frac{m}{\sqrt{n}}\sum_{i=1}^{n}g_1(X_i)$ is asymptotically normal and that the remaining terms converge to $0$ in probability.
$\newline$
If $(X_i)_{i\in\mathbb{N}}$ is strong mixing then this also applies to $(g_1(X_i))_{i\in\mathbb{N}}$, because $g_1$ is measurable (\cite{Koro1989}), and the mixing coefficients are smaller or equal to the original ones.
With these considerations 
 and observing that $(g_1(X_i))_{i\in\mathbb{N}}$ is strong mixing with mixing coefficients $\alpha(l)=O(l^{-\delta})$ for a $\delta>2$ and moreover $\mathbb{E}(g_1(X_i))=0$ and $g_1(X_i)$ is bounded (because $h$ is bounded) we can apply Theorem 1.6 of \cite{Ibra1961} getting $\sigma<\infty$ and
\begin{align*}
\frac{m}{\sqrt{n}}\sum_{i=1}^ng_1(X_i)\stackrel{D}{\longrightarrow}N(0,m^2\sigma^2).
\end{align*}

\vspace*{0.5cm}

It remains to show that the remaining terms of the Hoeffding decomposition are of order $o_P(1)$.
For this we apply Lemma \ref{finalcov} and show $\sum_{l=0}^nl\alpha^{\frac{\rho}{2\rho+1}}(l)=O(n^{\gamma})$ for a $\rho\geq 0$.

Using the assumption $\alpha(l)=O(l^{-\delta})$ for a $\delta>\tfrac{2\rho+1}{\rho}$ we get for a $\gamma<1$

\begin{align*}
\sum_{l=0}^nl\alpha^{\frac{\rho}{2\rho+1}}(l)\leq \sum_{l=1}^nl^{1-\delta{\frac{\rho}{2\rho+1}}}=O(n^{\gamma}).
\end{align*}

Now it is for all $2\leq k\leq m$
\begin{align*}
&\var\left(\sqrt{n}\binom{m}{k}\binom{n}{k}^{-1}\sum_{1\leq i_1<\ldots<i_k\leq n}g_k(X_{i_1},\ldots,X_{i_k})\right)\\
&\leq \frac{m^{2k}k^{\frac{k}{2}}}{n^{2k-1}}\\
&\phantom{\leq \frac{m}{n}}\sum_{1\leq i_1<\ldots <i_k\leq n}\sum_{1\leq i_{k+1}<\ldots<i_{2k}\leq n}\left|\mathbb{E}\left(g_k(X_{i_1},\ldots,X_{i_k})g_k(X_{i_{k+1}},\ldots, X_{i_{2k}})\right)\right|\\
&\leq \frac{m^{2k}k^{\frac{k}{2}}}{n^{2k-1}}\sum_{i_{1},\ldots,i_{2k}=1}^n\left|\mathbb{E}\left(g_k(X_{i_1},\ldots,X_{i_k})g_k(X_{i_{k+1}},\ldots, X_{i_{2k}})\right)\right|\\
&=O(n^{2k-2+\gamma-(2k-1)})=O(n^{-1+\gamma}).
\end{align*}

And so
\begin{align*}
 \var\left(\sqrt{n}\binom{m}{k}\binom{n}{k}^{-1}\sum_{1\leq i_1<\ldots<i_k\leq n}g_k(X_{i_1},\ldots,X_{i_k})\right)\stackrel{n\rightarrow \infty}{\longrightarrow}0
 \end{align*}
and with the Chebychev inequality we obtain
\begin{align*}
\sqrt{n}\binom{m}{k}\binom{n}{k}^{-1}\sum_{1\leq i_1<\ldots<i_k\leq n}g_k(X_{i_1},\ldots,X_{i_k})\stackrel{P}{\longrightarrow}0~\text{ for }n\rightarrow \infty.
\end{align*}

Using the Theorem of Slutsky we get the result of the theorem.  
\end{proof}

\vspace*{0.2cm}

\begin{proof}[Corollary \ref{glican}]
$\newline$
Using the the Hoeffding decomposition we obtain
\begin{align*}
&\sup\limits_{t\in\mathbb{R}}\left\vert\sqrt{n}\left(H_n(t)-H_F(t)\right)\right\vert\\
&=\sup\limits_{t\in\mathbb{R}}\left\vert\sqrt{n}\left(H_F(t)+\sum_{j=1}^{m}\binom{m}{j}\frac{1}{\binom{n}{j}}S_{jn,t}-H_F(t)\right)\right\vert\\
&=\sup\limits_{t\in\mathbb{R}}\big\vert\frac{m}{\sqrt{n}}\sum_{i=1}^{n}g_1(X_i,t)+\sqrt{n}\frac{\binom{m}{2}}{\binom{n}{2}}\sum_{1\leq i<j\leq n}g_2(X_i,X_j,t)\\
&+\ldots+\sqrt{n}\frac{1}{\binom{n}{m}}\sum_{1\leq i_1 <\ldots<i_m\leq n}h_m(X_{i_1},\ldots,X_{i_m},t)\big\vert\\
&\leq \sup\limits_{t\in\mathbb{R}}\left\vert\frac{m}{\sqrt{n}}\sum_{i=1}^{n}g_1(X_i,t)\right\vert+\sup\limits_{t\in\mathbb{R}}\left\vert\sqrt{n}\frac{\binom{m}{2}}{\binom{n}{2}}\sum_{1\leq i<j\leq n}g_2(X_i,X_j,t)\right\vert\\
&+\ldots+\sup\limits_{t\in\mathbb{R}}\left\vert\sqrt{n}\frac{1}{\binom{n}{m}}\sum_{1\leq i_1 <\ldots<i_m\leq n}h_m(X_{i_1},\ldots,X_{i_m},t)\right\vert.
\end{align*}
For the first summand we get, using Theorem \ref{invpri} and the Continuous Mapping theorem,
\begin{align*}
\sup\limits_{t\in\mathbb{R}}\left\vert\frac{m}{\sqrt{n}}\sum_{i=1}^{n}g_1(X_i,t)\right\vert\rightarrow \lVert W \rVert_{\infty}.
\end{align*}
Since $W$ is a continuous Gaussian process we have $\lVert W \rVert_{\infty}=O_p(1)$.

\vspace*{0.25cm}

For the remaining results we want to apply Lemma \ref{remterm}. Therefore the kernel of the $U$-process $g(x_1,\ldots,x_m,t)=1_{[h(x_1,\ldots,x_m)\leq t]}$ has to satisfy the extended uniform variation condition. This can be shown using the Lipschitz-continuity of $h$: 
\begin{align*}
&\sup\limits_{\lVert(x_1,\ldots,x_m)-({X'}_1,\ldots,{X'}_m)\rVert\leq \epsilon }\left\vert 1_{[h({X'}_1,\ldots,{X'}_m)\leq t]}-1_{[h({x}_1,\ldots,{x}_m)\leq t]}\right\vert\\
=&\begin{cases}
1 &\text{, if } h({X'}_1,\ldots,{X'}_m) \in \left(t-L\epsilon,
 t+L\epsilon\right)\\
0 &\text{, else}
\end{cases}
\end{align*}
and so
\begin{align*}
&\mathbb{E}\left(\sup\limits_{\lVert(x_1,\ldots,x_m)-({X'}_1,\ldots,{X'}_m)\rVert\leq \epsilon }\left\vert1_{[h({X'}_1,\ldots,{X'}_m)\leq t]}-1_{[h({x}_1,\ldots,{x}_m)\leq t]}\right\vert\right)\\
&\leq \sup\limits_{t \in \mathbb{R} }\left\vert\mathbb{E}\left(1_{[h({X'}_1,\ldots,{X'}_m)\in (t-L\epsilon,t+L\epsilon)]}\right)\right\vert\\
&\leq \sup\limits_{t \in \mathbb{R} }\left\vert\int_{t-L\epsilon}^{t+L\epsilon}h_F(x)dx\right\vert\leq 2L\epsilon(\sup\limits_{x\in \mathbb{R}}h_F(x))\leq L'\epsilon,
\end{align*}
since $h_F$ is bounded.

Using the arguments above we can also show that $g$ satisfies the extended uniform variation condition.
For arbitrary $2\leq k \leq m$ and $i_1<i_2<\ldots<i_m$
\begin{align*}
&\mathbb{E}\Biggl(\sup\limits_{\lvert x_1-{Y}_{i_1}\rvert\leq \epsilon }\Bigg\vert 1_{[h({Y}_{i_1},X_{i_2},\ldots,X_{i_k},Y_{i_{k+1}},\ldots,{Y}_{i_m})\leq t]}\\
&\phantom{mathbb{E}\Biggl(\sup\limits_{\lvert x_1-{Y}_{i_1}\rvert\leq \epsilon }\Bigg)}-1_{[h({x}_1,X_{i_2},\ldots,X_{i_k},Y_{i_{k+1}},\ldots,{Y}_{i_m})\leq t]}\Bigg\vert\Biggr)\\
&\leq \sup\limits_{t \in \mathbb{R} }\left\vert\int_{t-L\epsilon}^{t+L\epsilon}h_{F;X_{i_1},\ldots,X_{i_k}}(x)dx\right\vert\leq L\epsilon.
\end{align*}

Applying Lemma \ref{remterm} we get for $2\leq k \leq n$

\begin{align*}
&\sup\limits_{t\in\mathbb{R}}\left\vert\sqrt{n}\frac{\binom{m}{k}}{\binom{n}{k}}\sum_{1\leq i_1,\ldots,i_k\leq n}g_k(X_{i_1},\ldots,X_{i_k},t)\right\vert\\
&\leq \sqrt{n} n^{-k} o_p(n^{k-\frac{1}{2}-\frac{\delta-2}{8\delta}})=o_p(n^{-\frac{\delta-2}{8\delta}}).
\end{align*}

\vspace*{0.1cm}
With Slutsky's Theorem the proof is completed. 
\end{proof}

\section*{Acknowledgements}
\label{sec:8}
The financial support of the Deutsche Forschungsgemeinschaft (SFB 823, Statistical modelling
of nonlinear dynamic processes) is gratefully acknowledged.

\end{document}